\documentclass{amsart}

\usepackage{amsmath}
\usepackage{amssymb}
\usepackage{amscd}
\usepackage{hyperref}
 \usepackage{xypic}

\newtheorem{theorem}{Theorem}[section]
\newtheorem{thm}[theorem]{Theorem}

\newtheorem{lem}[theorem]{Lemma}

\theoremstyle{definition}
\newtheorem{defn}[theorem]{Definition}

\newtheorem{ques}[theorem]{Question}

\newtheorem{rem}[theorem]{Remark}

\newtheorem{conj}[theorem]{Conjecture}

\theoremstyle{remark}

\newcommand{\mbb}{\mathbb}
\newcommand{\QQ}{\mbb{Q}}

\newcommand{\ZZ}{\mbb{Z}}

\newcommand{\RR}{\mbb{R}}
\newcommand{\AAA}{\mbb{A}}
\newcommand{\PP}{\mbb{P}}

\newcommand{\mc}{\mathcal}

\newcommand{\mcL}{\mc{L}}

\newcommand{\mcT}{\mc{T}}

\newcommand{\mcX}{\mc{X}}

\newcommand{\OO}{\mc{O}}

\newcommand{\SP}{\text{Spec }}

\newsavebox{\sembox}
\newlength{\semwidth}
\newlength{\boxwidth}

\newsavebox{\semrbox}
\newlength{\semrwidth}
\newlength{\boxrwidth}

\newcommand{\tx}{\tilde{X}}

\title
{Towards the symplectic Graber-Harris-Starr theorems}

\author[Tian]{Zhiyu Tian}
\address{
Department of Mathematics \\
California Institute of Technology \\ 
Pasadena, CA, 91125}
\email{tian@caltech.edu}

\date{\today}

\begin{document}

%%%%%%%%%%%%%%%%%%%%%%%%%%%%%%%%%%%%%%%%%%%%%%%%%%%%%%%%%%%%%%%%%%%%
%%
%% Abstract
%%
%%%%%%%%%%%%%%%%%%%%%%%%%%%%%%%%%%%%%%%%%%%%%%%%%%%%%%%%%%%%%%%%%%%%

\begin{abstract}
 A theorem of Graber, Harris, and Starr states that a rationally connected fibration over a curve has a section. We study an analogous question in symplectic geometry. Namely, given a rationally connected fibration over a curve, can one find a section which gives a non-zero Gromov-Witten invariant? We observe that for any fibration, the existence of a section which gives a non-zero Gromov-Witten invariant only depends on the generic fiber, i.e. a variety defined over the function field of a curve. Some examples of rationally connected fibrations with this property are given, including all rational surface fibrations. We also prove some results, which says that in certain cases we can ``lift" Gromov-Witten invariants of the base to the total space of a rationally connected fibration. 
\end{abstract}

%%%%%%%%%%%%%%%%%%%%%%%%%%%%%%%%%%%%%%%%%%%%%%%%%%%%%%%%%%%%%%%%%%%%%%
%%
%% Body
%%
%%%%%%%%%%%%%%%%%%%%%%%%%%%%%%%%%%%%%%%%%%%%%%%%%%%%%%%%%%%%%%%%%%%%%%

\maketitle

%% \tableofcontents

%%%%%%%%%%%%%%%%%%%%%%%%%%%%%%%%%%%%%%%%%%%%%%%%%%%%%%%%%%%%%%%
%%
%% Section: Introduction
%% 
%%%%%%%%%%%%%%%%%%%%%%%%%%%%%%%%%%%%%%%%%%%%%%%%%%%%%%%%%%%%%%%

\section{Introduction}
This paper is devoted to the study of some symplectic geometric/topological aspects of rationally connected varieties. First recall that a smooth complex projective variety is called \emph{rationally connected} if for any pair of points there is a rational curve containing them. In the mid-nineties, after the introduction of Gromov-Witten invariants, Koll{\'a}r and Ruan discoved that many aspects of algebraic geometry should have their symplectic analogues. In particular, the following conjectures are proposed. 
\begin{conj}\label{conj:symrc}
Let $X$ be a smooth projective rationally connected variety. There is a non-zero Gromov-Witten invariant of the form $\langle [pt], [pt], A_1, \ldots, A_n, I \rangle^X_{0, \beta}$. Here $A_i \in H^*(X, \QQ)$ and $I$ is the pull-back of some cohomology class from $\overline{M}_{0, n+2}$ via the forgetful map $\overline{M}_{0, n+2}(X, \beta) \to \overline{M}_{0, n+2}$ (if $n \geq 1$).
\end{conj}

\begin{rem}
A symplectic manifold with the above property is called \emph{symplectic rationally connected}. One can also define symplectic rational connectedness as having some non-zero descendant Gromov-Witten invariant with two point insertions. But this imply that the variety is symplectic rationally connected in the sense of the above definition by recursively replacing $\psi$ classes with classes pulled back from $\overline{M}_{0, n+2}$.
\end{rem}

\begin{conj}[Koll{\'a}r, \cite{KollarUni}]\label{conj:Kollar}
Let $X$ and $X'$ be two smooth projective varieties which are symplectic deformation equivalent. Then $X$ is rationally connected if and only if $X'$ is.
\end{conj}
A stronger conjecture is also proposed by Koll{\'a}r in a private communication.

\begin{conj}[Koll{\'a}r]\label{conj:Kstrong}
The maximal rationally connected (MRC) quotient of a smooth projective variety is a symplectic deformation invariant.
\end{conj}

Here is an explanation of some of the terms above. For any smooth projective variety, the polarization gives a K{\"a}hler, thus symplectic, form. A variety $X$ is symplectic deformation equivalent to a variety $X'$, if there is a family of symplectic manifolds $(X_t, \omega_t)$, diffeomorphic to each other, such that $(X_0, \omega_0)$ (resp. $(X_1, \omega_1)$) is isomorphic to $(X, \omega)$ (resp. $(X', \omega')$) as a symplectic manifold. Note that different choices of the polarization give symplectic deformation equivalent varieties.

The MRC quotient of a variety is a rational dominant map $X \dashrightarrow Z$, whose restriction to an open dense subset is a proper morphism onto its image with rationally connected general fiber,  such that for any other rational dominant map $X \dashrightarrow Z'$ with above property, there is a rational map $Z' \dashrightarrow Z$ making the obvious diagram commutative.

One issue about Conjecture \ref{conj:Kstrong} is that it is difficult to make sense of a rational map being symplectic deformation invariant. However, we can at least ask that some numerical invariants, e.g. the dimension and the cohomology class of a general fiber, are symplectic deformation invariant.

Koll{\'a}r's conjecture \ref{conj:Kollar} is proved for $3$-folds in \cite{SymRC}. The stronger one also follows from the proof of that paper. And for all rationally connected surfaces and some classes of rationally connected $3$-folds, the existence of the non-vanishing Gromov-Witten invariant is known \cite{SymRC}.

In this paper we propose to study the following closely related questions.
\begin{ques}\label{ques:section}
Let $X \to C$ be a projective morphism from a smooth projective variety onto a smooth curve such that a general fiber is rationally connected (or symplectic rationally connected). Is there a non-zero Gromov-Witten invariant given by a section?
\end{ques}

\begin{ques}\label{ques:totalspace}
Let $X \to Y$ be a surjective morphism between smooth projective varieties. Assume both $Y$ and a general fiber are symplectic rationally connected. Is $X$ symplectic rationally connected?
\end{ques}

\begin{ques}\label{ques:uniruled}
Let $\pi: X \to Y$ be a morphism between smooth projective varieties. Assume a general fiber $F$ is symplectic rationally connected and $Y$ is uniruled. Is there a curve class $\beta$ in $H_2(X, \ZZ)$ with a non-zero Gromov-Witten invariant of the form $\langle [pt], \ldots \rangle^{X}_{0, \beta}$ such that $\pi_*(\beta)$ is non-zero?
\end{ques}

We now explain how these questions are related to the conjectures \ref{conj:symrc}, \ref{conj:Kollar}, and \ref{conj:Kstrong}. 

The following way to prove that every rationally connected variety has a non-vanishing Gromov-Witten invariant with two point insertions is proposed in \cite{SymRC}. Namely, first show that the existence of such an invariant is a birational invariant and then find in each birational class a ``good" representative which has such a Gromov-Witten invariant. By the minimal model program (MMP), most (in some sense) rationally connected varieties are birational to fibrations of rationally connected varieties over another rationally connected variety. Thus Question \ref{ques:totalspace} can be viewed as an inductive step to Conjecture \ref{conj:symrc}.

The idea behind Question \ref{ques:uniruled} and Conjecture \ref{conj:Kstrong} is that we should be able to construct some kind of ``symplectic MRC quotient", which is a symplectic deformation invariant. And we should be able to ``lift" the uniruled rational curve on the base of some rationally connected fibration and the process should be visible in terms of Gromov-Witten invariants. Therefore the symplectic MRC quotient should be non-uniruled and coincide with the usual MRC quotient.

The expectation is that the answers are always affirmative to the above questions. We would like to call these type of results the symplectic Graber-Harris-Starr theorems since they are the symplectic analogues of the celebrated theorem of Graber, Harris, and Starr, which says that a rationally connected fibration over a curve always has a section and, as an immediate corollary, the base of the MRC fibration is non-uniruled and the total space of a rationally connected fibration over a rationally connected variety is itself rationally connected. 

Note that the statements about MRC fibration and rational connectedness of the total space as above are really corollaries of the Graber-Harris-Starr theorem. However, it is not clear that Question \ref{ques:section}, if known to be true, will give positive answers to the other two questions. It is interesting to know if this is the case.
\begin{ques}
Do Questions \ref{ques:uniruled} and \ref{ques:totalspace} follow from Question \ref{ques:section}?
\end{ques}

For a proper fibration over a curve, the existence of a section is the same as the existence of a rational point in the generic fiber by the valuative criterion, which is certainly a property that only depends on the generic fiber. In this paper we first observe that the same is true in the symplectic setting.

\begin{thm}\label{thm:generic}
Let $\pi: X \to C$ (resp. $Y \to C$) be a fibration over a smooth projective curve with a smooth projective total space, and let $X_\eta$ (resp. $Y_\eta$) be the generic fiber. Assume that $X_\eta$ and $Y_\eta$ are isomorphic over the function field of $C$. Then there is a section of $X \to C$ which gives a non-zero Gromov-Witten invariant if and only if there is a section of $Y \to C$ giving a non-zero Gromov-Witten invariant of $Y$.
\end{thm}

Given this theorem, it is natural to ask the following.

\begin{ques}
Let $\pi: X \to C$ (resp. $Y \to C$) be a fibration over a smooth projective curve and $X_\eta$ (resp. $Y_\eta$) the generic fiber. Assume that $X_\eta$ and $Y_\eta$ are birational to each other over the function field of $C$. Then is it true that there is a section of $X \to C$ which gives a non-zero Gromov-Witten invariant if and only if there is a section of $Y \to C$ giving a non-zero Gromov-Witten invariant of $Y$?
\end{ques}

Given Theorem \ref{thm:generic}, one can ask for which rationally connected variety over the function field of a curve the symplectic Graber-Harris-Starr theorem holds. The following result gives a partial answer to this.

\begin{thm}\label{thm:section}
Let $\pi: \mcX \to C$ be a rationally connected fibration over a smooth projective curve $C\cong \PP^1$ with smooth projective total space. And let $\mcX_\eta$ be the generic fiber. Then there is a section of $\pi: \mcX \to C$, whose curve class gives a non-zero Gromov-Witten invariant of $\mcX$ if $\mcX_\eta$ has a fibration structure $$\mcX_\eta=\mcX_n \to \mcX_{n-1} \to \ldots \to \mcX_1 \to \SP K(C),$$ where the generic fiber of each morphism is a rationally connected variety of dimension at most $2$.
\end{thm}

As for Question \ref{ques:totalspace} and \ref{ques:uniruled}, we need some geometric assumptions on the base.

\begin{defn}\label{def:enumerative}
A Gromov-Witten invariant $\langle A_1, \ldots, A_n \rangle^X_{0, \beta}$ is \emph{enumerative} if 
\begin{enumerate}
\item  For any fixed subvarieties $Z_1, \ldots, Z_s$, the classes $A_i \in H^*(X, \ZZ)$ can be represented by irreducible subvarieties $Y_i$ of $X$ such that the intersection $Y_i \cap Z_j$ is empty unless the intersection product $A_i \cdot [Z_j]=[Y_i] \cdot [Z_j]$ is non-zero.
\item There is a way of choosing the representatives of $A_i$ such that the only curves in class $\beta$ that can meet all the representatives are \emph{embedded} irreducible rational curves.
\end{enumerate}
\end{defn}

\begin{rem}

\begin{enumerate}
\item By choosing the subvariety $Z$ to be a point, the first condition implies that the subvarieties $Y_i$ should be ``moving" and have no ``base point". In applications, we usually choose the constraints to be a point or the intersection of some base-point-free (or even very ample) divisors.
\item The embeddedness in the second condition is the key requirement. There are examples of Gromov-Witten invariants which count the number of irreducible \emph{immersed} rational curves. These kind of invariants should also be considered as ``enumerative". However, for technical reasons, the proof will be greatly simplified if we work with embedded rational curves. The most interesting case to us would be the case where there are two point classes in the constraints (i.e. symplectic rational connectedness). In dimension at least $3$, a general such curve is embedded. In dimension $2$, one can always make such a choice.
\item The definition in particular implies that the Gromov-Witten invariant is the same as the number of rational curves satisfying the incidence relations if the constraints are in general position, hence the name.
\end{enumerate}
\end{rem}

\begin{thm}\label{thm:totalspace}
Let $\pi: X \to Y$ be a morphism between smooth projective varieties. Assume that $Y$ has a non-zero enumerative Gromov-Witten invariant of the form $\langle [pt], [pt], A_1, \ldots, A_n \rangle^Y_{0, \beta}$. Assume the generic fiber $X_\eta$ (over the function field $K(Y)$ of $Y$) has a fibration structure $$X_\eta=X_n \to X_{n-1} \to \ldots \to X_1 \to \SP K(Y),$$ where the generic fiber of each morphism is a rationally connected variety of dimension at most $2$.
Then $X$ is symplectic rationally connected.
\end{thm} 

\begin{rem}
\begin{enumerate}
\item Here are some examples of $Y$ satisfying the condition in this theorem.
\begin{itemize}
\item $Y$ is a rationally connected surface.
\item $Y$ is a Fano $3$-fold or a rationally connected $3$-fold with Picard number $2$ (\cite{SymRC}).
\item $Y$ is a rational homogeneous space.
\end{itemize}
For each rationally connected $3$-fold $Y$, unless the minimal model program only produces a birational model of $Y$ as a $\QQ$-factorial, non-Gorenstein, terminal Fano $3$-fold of Picard number $1$, there is another $3$-fold $Y'$ which is birational to $Y$ and satisfies the condition. Assuming the smooth locus of a $\QQ$-factorial terminal Fano $3$-fold of Picard number $1$ is rationally connected (which is a conjecture in birational geometry), then such a $Y'$ also exists. Thus a fibration in the theorem over a $3$-dimensional base is, birationally, symplectic rationally connected. 

\item It is in general an open question whether symplectic rational connectedness is a birational invariance. However in the special cases of Theorem \ref{thm:totalspace}, we are able to show a special kind of birational invariance based on some geometric arguments.
\end{enumerate}
\end{rem}

\begin{thm}\label{thm:uniruled}
Let $X$ be a smooth projective variety and let $X \dashrightarrow \Sigma$ be a dominant rational map which is a proper morphism when restricted to an open dense subset. Assume the generic fiber $X_\eta$ has a fibration structure $$X_\eta=X_0 \to X_1 \to \ldots \to X_n \to \SP K(\Sigma),$$ where the generic fiber of each morphism is a rationally connected variety of dimension at most $2$. Also assume $\Sigma$ is uniruled and has a projective birational model which has one minimal uniruled class given by embedded rational curves. Then there is a curve class $\beta$ in $H_2(X, \ZZ)$ with a non-zero Gromov-Witten invariant of the form $\langle [pt], \ldots \rangle^{X}_{0, \beta}$. Furthermore, the image of a curve of class $\beta$ in $\Sigma$ is not a point.
\end{thm}

\begin{rem}
\begin{enumerate}
\item The assumption on the rational map $X \dashrightarrow Z$ naturally appears in the MRC quotient. 
\item The assumption on the minimal uniruled class holds if $\dim Z=2$, but not in higher dimensions. However, there might be ways to get rid of this assumption, which is added here to simplify the comparison of Gromov-Witten invariants of the ambient space and the subvariety (c.f. Remark \ref{rem:immersedcase}).
\item The reappearance of the conditions of the fiber in Theorems \ref{thm:section}, \ref{thm:uniruled}, and \ref{thm:totalspace} is not a coincidence. We essentially reduce the problem to the case of Theorem \ref{thm:section}. That is, in our situation, Question \ref{ques:totalspace} and \ref{ques:uniruled} follows from Question \ref{ques:section} and some manipulations of Gromov-Witten invariants. In general it is difficult to compare Gromov-Witten invariants of the ambient space with those of subvarieties. And the geometric assumptions make such a comparison possible.

\end{enumerate}
\end{rem}

\textbf{Acknowledgments:} The author would like to thank Tom Graber, Jason Starr and Aleksey Zinger for many helpful discussions. The connection between Conjecture \ref{conj:Kstrong} and  Question \ref{ques:uniruled} was pointed out by J{\'a}nos Koll{\'a}r. 

\section{Comparison of Gromov-Witten invariants via the degeneration formula}
In this section we recall the main technical result of Hu-Li-Ruan \cite{HLR}, which allows one to compare Gromov-Witten invariants of a variety and its blow-up. Then apply it to prove Theorem \ref{thm:generic}. Readers familiar with their results can safely skip to the last subsection. 

For the reader's convenience, we briefly summarize the comparison result we need from \cite{HLR}. In general one can write down the degeneration formula to compare such invariants. However the formula usually involves a sum of many terms, which is difficult to analyze. The result in \cite{HLR} says that only further degenerations of the curve have non-trivial contributions in the sum. They introduced a partial ordering on Gromov-Witten invariants, which measures the extent of the degeneration of the curve. Also the degeneration formula gives an \emph{invertible} linear map between the space of absolute and relative Gromov-Witten invariants. In particular, if one starts with a non-zero Gromov-Witten invariant, one is guaranteed to get a non-zero Gromov-Witten invariant, although the exact form of the two invariants are (in general) different and no so easy to see directly.

\subsection{Descendant GW-invariants, Relative GW-invariants, and the Degeneration formula}
In this section we recall some variants of Gromov-Witten invariants.

\begin{defn}\label{def:GW}
Let $\overline{\mathcal{M}}_{g, n}^{X, \beta}$ be the moduli stack of genus $g$, $n$-pointed stable maps to $X$ whose curve class is $\beta$. Let $\mcL_i$ be the line bundle on $\overline{\mathcal{M}}_{g, n}^{X, \beta}$ whose fiber over each point $(C, p_1, \ldots, p_n)$ is the restriction of the sheaf of differentials of the curve $C$ to the point $p_i$. Let $\psi_i$ be the first Chern class of $\mcL_i$. Then the descendant Gromov-Witten invariant is defined as 
\[
\langle \tau_{k_1}\gamma_1, \ldots, \tau_{k_n}\gamma_n \rangle^{X}_{g, \beta}=\int_{[\overline{\mathcal{M}}_{g, n}^{X, \beta}]^{\text{virt}}} \prod_i \psi_i^{k_i} ev_i^*\gamma_i,
\]
where $ev_i$ is the evaluation map given by the $i$-th marked point, and $\gamma_i \in H^*(X, \QQ)$.

For $\langle \tau_{k_1}\gamma_1, \ldots, \tau_{k_n}\gamma_n \rangle^{X}_{g, \beta}$, we can associate a decorated graph $\Gamma$ of one vertex decorated by $\beta$ and a tail for each marked points, decorated by $(k_i, \gamma_i)$. The resulting graph $\Gamma(\{(k_i, \gamma_i)\})$ is called a \emph{decorated weighed graph}.
\end{defn}

Next we discuss \emph{relative} Gromov-Witten invariants, which were first introduced in the symplectic category by Li-Ruan~\cite{LiRuan} and in the algebraic category by Jun Li~\cite{LiJun1}, \cite{LiJun2}. We will not recall the precise definition here since it is not needed. The reader may refer to the above-mentioned papers for more details.

Intuitively, the relative Gromov-Witten invariants count the number of stable maps satisfying certain incidence constraints and having prescribed tangency conditions with a given divisor. Let $X$ be a smooth projective variety and $D \subset X$ be a smooth divisor. Fix a curve class $\beta$ such that the intersection number $D \cdot \beta =m$ is non-negative. The relative Gromov-Witten invariants are not defined if the number $D \cdot \beta$ is negative. Also choose a partition $\{m_i, i=1, 2, \ldots, s\}$ of $m$. Then the relative Gromov-Witten invariants count the number of stable maps $f: (C, p_1, \ldots, p_r, q_1, \ldots, q_s) \rightarrow X$ with $r+s$ marked points such that the first $r$ points (absolute marked points) are mapped to cycles in $X$ and the last $s$ points (relative marked points) are mapped to some cycles in $D$ and $f^*D=\sum m_i q_i$. We can also define descendant relative Gromov-Witten invariants. We write such invariants as 
\[
\langle \tau_{d_1}\gamma_1, \ldots, \tau_{d_r}\gamma_r \vert (m_1, \delta_1), \ldots, (m_s, \delta_s)\rangle^{(X, D)}_{g, \beta}
\] 
where $\gamma_i \in H^*(X, \QQ), \delta_j \in H^*(D, \QQ)$. We also use the abbreviation 
\[
\langle \Gamma \{(d_i, \gamma_i)\} \vert {\mcT}_s \rangle^{X, D}_{g, \beta}
\]
 following  \cite{HLR}, where 
\[
\mcT_s=\{(m_1, \delta_1), \ldots, (m_s, \delta_s)\}
\]
is called the \emph{weighted partition}. In the degeneration formula, we have to consider stable maps from disconnected domains. The corresponding relative invariants are defined to be the product of those of stable maps from connected domains. Such invariants are denoted by 
\[
\langle \Gamma^\bullet \{(d_i, \gamma_i)\} \vert \mathcal{T}_s \rangle^{X, D}_{g, \beta}.
\]
We use $\bullet$ to indicate that the invariant is for a disconnected curve as \cite{HLR} and \cite{MP}. Finally we note that we can represent these invariants by decorated weighted graphs (c.f. Section 3.2 in \cite{HLR}), which is the disjoint union of the graphs described in Definition \ref{def:GW}.

Now we describe the degeneration formula. Let $W \rightarrow S$ be a projective morphism from a smooth variety to a pointed curve $(S, 0)$ such that a general fiber is smooth and connected and the fiber over $0$ is the union of two smooth irreducible varieties ($W^+, W^-$) intersecting transversely at a smooth subvariety $Z$. Let $\gamma_i$ be cohomology classes in a general fiber. Assume that the specialization of $\gamma_i$ in $W_0$ can be written as $\gamma_i(0)=\gamma_i^++\gamma_i^-$, where $\gamma_i^+ \in H^*(W^+, \QQ)$ and $\gamma_i^- \in H^*(W^-, \QQ)$. 

We first specify a map from a curve of genus $g$ to $W^+\cup W^-$ with the following properties:
\begin{enumerate}
\item [(i)] Each connected component is mapped to either $W^+$ or $W^-$ and carries a degree $2$ homology class;
\item [(ii)] The marked points are not mapped to $Z$;
\item [(iii)] Each point mapped to $Z$ carries a positive integer representing the order of the tangency.
\end{enumerate} 
The above data gives two graphs describing relative stable maps from possibly disconnected domains to $(W^+, Z)$ and $(W^-, Z)$, the graph of which are denoted by $\Gamma^\bullet_+$ and $\Gamma^\bullet_-$. From (iii) we get two partitions $T_+$ and $T_-$. Call the above data a \emph{degenerate genus $g$ $(\beta, l)$ graph} if the resulting pairs $(\Gamma^\bullet_+, T_+)$ and $(\Gamma^\bullet_-, T_-)$ satisfies the following: the total number of marked points is $l$, $T_+=T_-$, and the identification of relative tails produces a connected graph of $W$ with total homology class $[\beta]$ and genus $g$.

Denote by $Aut(T_k)$ the automorphism group of such partitions. Let $\{\delta_i\}$ be a self-dual basis of $H^{*}(Z, \QQ)$. By (iii), we have a weighted partition $\mcT_k=\{(t_j, \delta_{a_j})\}$ and its dual partition $\check{\mcT}_k=\{(t_j, \check{\delta}_{a_j})\}$, where $\check{\delta}_{a_j}$ is the Poincar{\'e} dual of $\delta_{a_j}$. Let $\beta_+$ (resp. $\beta_-$) be the total homology class of the curves mapped to $W^+$ (resp. $W^-$) in a degenerate $(\beta, l)$ graph. Then $\beta=\beta_+ +\beta_-$. The degeneration formula expresses the Gromov-Witten invariants of a general fiber in terms of the relative Gromov-Witten invariants of the degeneration in the following way:
\[
\langle \prod_i \tau_{d_i} \gamma_i \rangle^{W_t}_{g, \beta}=\sum \Delta(\mcT_k) \langle \Gamma^\bullet \{(d_i, \gamma_i^+)\} \vert \mathcal{T}_k \rangle^{W^+, Z}_{g_+, \beta_+} \langle \Gamma^\bullet \{(d_i, \gamma_i^-)\} \vert \check{\mathcal{T}}_k \rangle^{W^-, Z}_{g_-, \beta_-},
\]
where the summation is taken over all possible degenerate genus $g$ $(\beta, l)$ graphs, and
\[
\Delta(\mcT_k)= \vert Aut(T_k)\vert \cdot \prod_j t_j.
\]
By convention, if $\beta_+$ or $\beta_-$ is the zero homology class, the relative invariant is defined to be $1$.

In this paper we are mainly interested in the following special case of such degenerations: the deformation to the normal cone. Namely, let $X$ be a smooth projective variety and $S \subset X$ be a smooth subvariety. Then we take $W$ to be the blow-up of $X \times \AAA^1$ with blow-up center $S \times 0$. In this case, $W^- \cong \tx$, the blow-up of $X$ along $S$, and $W^+\cong \PP_S(\OO \oplus N_{S/X})$.

\subsection{The blow-up/blow-down correspondence}

Let $X$ be a smooth projective variety and $S\subset X$ be a smooth subvariety of codimension $k$. Denote by $\tilde{X}$ the blow-up of $X$ along $S$ and by $E$ the exceptional divisor. Here we allow $S$ to be a codimension $1$ subvariety, i.e. a divisor. In this case $\tilde{X}$ is isomorphic to $X$ and $E$ is isomorphic to $S$.

Let $\theta_1, \theta_2, \ldots, \theta_{m_S} \in H^*(S, \QQ)$ be a self dual basis of $S$. We now describe a basis of the exceptional divisor $E$. Note that $E=\PP_S(N_{S/X})$ is a $\PP^{k-1}$-bundle over $S$. Let $\lambda$ be the first Chern class of the relative $\OO(-1)$ bundle over $\PP_S(N_{S/X})$. Then the cohomology classes
\[
\pi_S^* \theta_i \cup \lambda^j, ~1 \leq i \leq m_S, ~0 \leq j \leq k-1,
\]
form a basis of $E$. Denote it by $\Theta=\{\delta_i \}$.

\begin{defn}
A \emph{standard (relative) weighted partition $\mu$} is a partition
\[
\mu=\{(\mu_1, \delta_{d_1}), \ldots, (\mu_{l(\mu)}, \delta_{d_{l(\mu)}})\},
\]
where $\mu_i$ and $d_i$ are positive integers with $d_i \leq k m_S$. $l(\mu)$ is called the \emph{length} of the partition.

For $\delta=\pi_S^* \theta \cup \lambda^j \in H^*(E, \QQ)$, with $j \leq k-1$, define
\[
\deg_S(\delta)=\deg \theta, \deg_f(\delta)=2j.
\]

For a standard weighted partition $\mu$, define
\[
\deg_S(\mu)=\sum_{i=1}^{l(\mu)} \deg_S (\delta_{d_i}), \deg_f(\mu)=\sum_{i=1}^{l(\mu)}\deg_f(\delta_{d_i}).
\] 

\end{defn}

Let $\sigma_1, \ldots, \sigma_{m_X}$ be a basis of $H^*(X, \QQ)$. Then the set of cohomology classes
\[
\gamma_j=\pi^* \sigma_j, 1 \leq j \leq m_X,
\]
\[
\gamma_{j+m_X}=\iota_*(\delta_j), 1 \leq j \leq k m_S
\]
generate a basis of $H^(\tilde{X}, \QQ)$, where $\pi: \tilde{X} \to X$ is the blow-up along $S$, $\iota : E \rightarrow \tilde{X}$ is the inclusion and $\iota_*$ is the induced Gysin map.

\begin{defn}
A \emph{connected standard relative Gromov-Witten invariant} of $(\tilde{X}, E)$ is of the form
\[
\langle \omega \vert \mu \rangle^{\tilde{X}, E}_{0, A}=\langle \tau_{k_i}\gamma_{L_1}, \ldots, \tau_{k_n}\gamma_{L_n} \vert \mu \rangle^{\tilde{X}, E}_{0, A},
\]
where $A$ is an effective curve class on $\tx$, $\mu$ is a standard weighted partition with $\sum \mu_j = E \cdot A$, and $\gamma_{L_i}=\pi^* \sigma_{L_i}$.
\end{defn}

We write $\Gamma(\omega) \vert \mu$ for the decorated graph of such invariants.

We now relate absolute invariants of $X$ to relative invariants of $\tilde{X}$.

To a relative insertion $(m, \delta)$ with $\delta=\pi_S^*\theta_i \cup \lambda^j$, we associate the absolute insertion $\tau_{d(m, \delta)}(\tilde{\delta})$, where
\[
\tilde{\delta}=\iota_*(\theta_i), d(m, \delta)=km-k+j.
\]
Given a weighted partition $\mu=\{(\mu_i, \delta_{k_i})\}$, we define 
\[
d_i(\mu)=d(\mu_i, \delta_{k_i})=k \mu_i-k+\frac{1}{2}\deg_f(\delta_{k_i}),
\]
\[
\tilde{\mu}=\{\tau_{d_1(\mu)}(\tilde{\delta}_{k_1}), \ldots, \tau_{d_{l(\mu)}(\mu)}(\tilde{\delta}_{k_{l(\mu)}})\}.
\]
Given a standard relative invariant $\langle \Gamma^\bullet(\omega) \vert \mu\rangle^{\tilde{X}, E}$, we define the absolute descendant invariant associated to the relative invariant to be 
\[
\langle {\Gamma}^\bullet(\omega, \tilde{\mu})\rangle^X
\]
Here all the insertions $\omega$ in the relative invariants are of the form $\pi^* \sigma_i$; the corresponding insertions in the absolute invariants are just $\sigma_i$.

\begin{defn}
An absolute descendant invariant of $X$ is called a \emph{colored absolute descendant invariant relative to $S$} if its insertions are divided into two collections $\omega$ and $\tilde{\mu}$ such that each insertion in $\omega$ is of the form $\tau_{d_i}\sigma_i$ and each insertion in $\tilde{\mu}$ is of the form $\tau_{d_k} \tilde{\delta}_k$. 
\end{defn}

\begin{rem}
An absolute invariant may give different colored invariants depending on how one groups the insertions.
\end{rem}

\begin{defn}
If $k=1$, then a colored absolute descendant invariant of $X$ relative to $S$ (with curve class $\beta$) is called \emph{admissible} if $\sum \mu_j = E \cdot \beta$.
\end{defn}

The following lemma is essentially Lemma 5.14 in \cite{HLR}. Note that in their paper they only consider the case of primary Gromov-Witten invariants. But the proof is actually the same.
\begin{lem}\label{lem:bijection}
If $\mu \neq \mu'$, then $\tilde{\mu} \neq {\tilde{\mu}}'$. Therefore there is a natural bijection between the set of colored weighted absolute graphs relative to $S$ and the set of weighted relative graphs in $\tilde{X}$ relative to $E$ if $k>1$. The same is true if we restrict to the admissible ones when $k=1$.
\end{lem}

\begin{rem}
Notice that different relative invariants may give the same absolute invariants. But these absolute invariants are different as colored absolute invariants.
\end{rem}

Let $I$ be the set of standard weighted relative graph $\Gamma^\bullet  (\omega) \vert \mu$. 

Define $\RR^I_{\tilde{X}, E}$ to be an infinite dimensional vector space. A standard weighted relative invariant $\langle \Gamma^\bullet (\omega) \vert \mu \rangle^{\tilde{X}, E}$ gives a vector $v_{\tx, E}$ in $\RR^I_{\tilde{X}, E}$. By Lemma~\ref{lem:bijection}, $I$ is also the set of colored standard weighted absolute graphs relative to $S$. Thus we also have an infinite dimensional vector space $\RR^I_{X, S}$. Similarly, an absolute invariant $\langle \Gamma^\bullet (\omega, \tilde{\mu})\rangle^X$ gives a vector $v_{X, S}$ in this vector space.

We can now state the theorem on blow-up/blow-down correspondence  .

\begin{thm}[\cite{HLR}, Theorem 5.15]\label{thm:Correspondence}
Let $\pi: \tilde{X} \rightarrow X$ be the blow-up of a smooth projective variety along a smooth center $S$. Then there is a partial ordering on $I$ such that if we use the same partial ordering on the coordinates of $\RR^I_{\tilde{X}, E}$ and $\RR^I_{X, S}$, then the linear map
\[
A_S: \RR^I_{\tilde{X}, E} \rightarrow \RR^I_{X, S},
\]
given by the degeneration formula such that $A_S(v_{\tx, E})=v_{X, S}$, is a lower triangular map and $A_S$ only depends on $S$ and its normal bundle.
\end{thm}

\subsection{Proof of Theorem \ref{thm:generic}}
Since the generic fiber $X_\eta$ and $Y_\eta$ are isomorphic, there is a birational isomorphism between the total space $X$ and $Y$ that is a morphism in a neighborhood of the generic fiber. By the weak factorization theorem (Theorem 0.1.1, \cite{AKMW}), we can factorize the birational isomorphism by a number of blow-ups and blow-downs whose centers are supported in special fibers. Thus it suffices to prove the theorem in such special cases.

In the following, assume that $Y$ is the blow-up of $X$ along a smooth subvariety $S$ that is supported in a fiber. We will say that $X$ (resp. $Y$) satisfies the symplectic Graber-Harris-Starr (SGHS) condition if there is a non-zero Gromov-Witten invariant of $X$ (resp. $Y$) whose curve class is a section.

First assume that $X$ satisfies the SGHS condition. By linearality of Gromov-Witten invariants, we may also assume the invariant is of the form
\[
\langle \tau_{d_1}\sigma_1, \ldots, \tau_{d_n} \sigma_n \rangle^{X}_{g, \beta},
\]
where $\beta$ is a section class (i.e. has intersection number $1$ with a fiber). We degenerate $X$ into $Y\cong \text{Bl}_S X$ and $\PP_S(\OO \oplus N_{S/\mcX})$ and apply the degeneration formula. So there is a non-zero relative invariant:
\[
\langle \tau_{d_1}\sigma_1, \ldots, \tau_{d_k} \sigma_k \vert \mu \rangle^{Y, E}_{g, \beta}.
\]
Then apply Theorem \ref{thm:Correspondence} to $(\text{Bl}_E Y, E)$ and $(Y, E)$. Note that the blow-up of $Y$ with center $E$ is $Y$ itself. So Theorem \ref{thm:Correspondence} gives a non-zero absolute invariant of $Y$. Notice that we only blow up a subvariety in a special fiber. Thus there are always curves in the $(Y, E)$ side in the degeneration formula and the curve class is a section since the intersection number with a fiber is $1$.

Conversely, suppose that there is a non-zero descendant Gromov-Witten invariant on $Y$ of the form 
\[
\langle \tau_{d_1}\gamma_1, \ldots, \tau_{d_n} \gamma_n \rangle^{Y}_{g, \tilde{\beta}}.
\]
We may assume that $\gamma_i, 1 \leq i \leq m$ are of the form $\pi^* \sigma_{j_i}$ and $\gamma_i, m+1 \leq i \leq n$ are of the form $\iota_*(\delta_{j_i})$. Then we degenerate $Y$ into $Y$ and $\PP_E(\OO \oplus N_{E/Y})$ and specialize $\gamma_i$ ($m+1 \leq i \leq n$) to the projective bundle side. Then there is a non-zero relative invariant of the form:
\[
\langle \tau_{d_1}\gamma_1, \ldots, \tau_{d_k} \gamma_k \vert \mu \rangle^{Y, E}_{g, \beta},
\]
with $k \leq m$. In particular, all the $\gamma_i, 1 \leq i \leq k$ are of the form $\pi^* \sigma_{j_i}$. Again apply Theorem \ref{thm:Correspondence} to $(Y, E)$ and $(X, S)$. By the same observation above, we get a non-zero absolute descendant invariant of desired form.

\section{Proof of the main theorems}
\subsection{Fano fibrations}In this subsection, let $Z$ be a normal projective variety with at worst terminal singularities. And let $\pi: Z\rightarrow B$ be a contraction of some $K_Z$-negative extremal face to a smooth projective curve $B$ of genus $0$. Assume a general fiber $F$ is a smooth Del Pezzo surface, or $\PP^1$. When a general fiber is $\PP^1$, we allow special fibers to be the union of two smooth $\PP^1$'s. Let $f: \tilde{Z} \rightarrow Z$ be a resolution of singularities that is isomorphic near a general fiber. Note that all the exceptional divisors are supported in special fibers of $\pi \circ f: \tilde{Z} \rightarrow Z \rightarrow B$. The following is Thereom 6.1 and Proposition 2.2 in \cite{SymRC}.

\begin{thm}\label{thm:fano}
There is a non-zero enumerative Gromov-Witten invariant of the form $$\langle \underbrace{[pt], \ldots, [pt]}_{m\geq 2}, \ldots\rangle_{0, \beta}^{\tilde{Z}}$$ for some class $\beta$ which is a section of the fibration $\tilde{Z} \to B$. 
\end{thm}

\begin{rem}\label{rem}
The proof in \cite{SymRC} actually shows a little more. Namely, if we have a curve class which has the same intersection number with the divisors $K_{\tilde{Z}}$, and $K_{\tilde{Z}}-f^*K_Z=\sum a_i E_i$ and a fixed ample divisor, then any curve in that curve class that can meet all the constraints is necessarily irreducible. This observation is crucial for the applications in the next section.
\end{rem}

\subsection{Proof of Theorems \ref{thm:section}, \ref{thm:totalspace}, and \ref{thm:uniruled}} The general idea of the proofs is to reduce all the cases to the situation of Theorem \ref{thm:fano} by running the minimal model program. Then we need to compare Gromov-Witten invariants of a subvariety with the ambient space and between different birational models. The general technique of applying degeneration formula is not enough to do the comparison, which explains the geometric assumptions in the statements of the main theorems. We begin with some technical lemmas.
\begin{lem}\label{lem:birational}
Let $f: \tilde{Y} \to Y$ be a birational morphism between smooth projective varieties. Assume $Y$ has a non-zero genus $0$ enumerative Gromov-Witten invariant of the form
\[
\langle [pt], [A_1], \ldots, [A_n] \rangle^Y_{0, \beta}.
\]
 Then $X$ also has a non-zero genus $0$ enumerative Gromov-Witten invariant of the form
\[
\langle [pt], f^*[A_1], \ldots, f^*[A_n] \rangle^{\tilde{Y}}_{0, \beta},
\]
and $f_*(\beta')=\beta$.
\end{lem}

\begin{proof}
Let $C$ be a rational curve which gives the enumerative Gromov-Witten invariant on $Y$ (i.e. $[C]=\beta$). Then a general deformation of $C$ is a free curve and can be deformed away from any fixed codimension $2$ locus. Then we can choose the constraints to be general such that any curve meeting all the constraints are disjoint from the exceptional locus in $Y$. Take the inverse images of these curves in $\tilde{Y}$ and choose the constraints to be the pull-back of the constraints. It it easy to check that the pull-back classes also satisfy the first condition in Definition \ref{def:enumerative}. Then any curve on $\tilde{Y}$ that can meet all the constraints is necessarily one of the inverse images of the curves in $Y$. These curves give the non-zero genus $0$ enumerative Gromov-Witten invariant as stated.
\end{proof}

\begin{lem}\label{lem:mfs}
Let $Y$ be a smooth projective variety and let $X$ be a projective variety with terminal singularities. Assume $f: X \to Y$ is a Mori fiber space whose general fiber is a del Pezzo surface or $\PP^1$. Let $g: \tilde{X} \to X$ be a resolution of singularities which is isomorphic over the smooth locus of $X$. Assume that there is a non-zero enumerative Gromov-Witten invariant of the form 
\[
\langle \underbrace{[pt], \ldots, [pt]}_k, [A_1], \ldots, [A_n]\rangle^Y_{0, [C]}
\] 
on $Y$. Then there is a non-zero enumerative Gromov-Witten invariant of the form 
\[
\langle \underbrace{[pt], \ldots, [pt]}_k, (f\circ g)^*[A_1], \ldots, (f\circ g)^*[A_n]\ldots \rangle^{\tilde{X}}_{0, \beta}, 
\]
and $(f\circ g)_*(\beta)=[C]$.
\end{lem}

\begin{proof}
We consider the family of stable genus $0$ maps to $Y$ of class $[C]$ and meeting general representatives of $A_1, \ldots, A_n$ and $k-1$ general points, with $n+k-1$ sections $S_i: V \to U$. 
\[
\begin{CD}
 U @>ev>> Y\\
@V \pi VV\\
V
\end{CD}
\]
We will only consider the family whose general points parameterize irreducible free rational curves and for which the evaluation map $ev: U \to {Y}$ is dominant, since other components will not contribute to the Gromov-Witten invariant. We have the following diagram where each square is a fiber product, and $p$ is a general point in $Y$. 
\[
\begin{CD}
\tilde{Z'} @>>>\tilde{W}    @>>> \tilde{X}\\
@VVV @V g' VV @V g VV\\
Z @>>> W    @>>> X\\
 @VVV @Vf'VV @VfVV\\
C@>>>U @>ev>> Y\\
@VVV @V\pi VV\\
\pi(ev^{-1}(p))@>>>V\\
\end{CD},
\]
We claim that $Z=(\pi \circ f')^{-1}(\pi(ev^{-1}(p)))$ has only terminal singularities and $\tilde{Z}=(\pi \circ f' \circ g')^{-1}(\pi(ev^{-1}(p)))$ is smooth. In addition, all the singularities of $Z$ come from singularities of $X$. Indeed, blow up the subvarieties $A_i$ in $Y$ and consider the universal family of rational curves which are strict transforms of curves meeting all the constraints. Note that a general curve in this family is contained in the smooth locus of the blow-up by the enumerative assumption. Such a curve contains a very general point, thus free. Then by a similar argument as Proposition 3.5, Chap. II, \cite{Kollar96}, the evaluation map (to the blow-up) is {\'e}tale along a general curve. Thus the locus in $U$ where $U \to Y$ is not {\'e}tale and dominates $V$ are precisely the sections $S_i: V \to U$. But for a general curve in the family $V$, the marked points are mapped to points not in the discriminant locus of $X \to Y$. In particular, the fibers over these points are smooth and the $W \to X$ is {\'e}tale along singular fibers over a general curve. Since a general fiber of a fibration whose total space has terminal singularities (resp. is smooth) has at worst terminal singularities (resp. is smooth), the claim follows.

Furthermore, we have
\[
K_Z =K_{W/V} \vert_Z=K_C+K_{W/U}\vert_Z=K_C+K_{X/Y}\vert_Z,
\]
\[
K_{\tilde{Z}} =K_{\tilde{W}/V} \vert_{\tilde{Z}}=K_C+K_{\tilde{W}/U}\vert_Z=K_C+K_{\tilde{X}/Y}\vert_Z,
\]
\[
K_{\tilde{X}/Y}=K_{X/Y} + \sum a_i E_i,
\]
where $E_i$'s are exceptional divisors of $\tilde{X} \to X$. Thus
\[
K_{\tilde{Z}}=K_Z+ (\sum a_i E_i)\vert_Z.
\]
Let $L$ be an ample divisor on $Y$ such that $-K_X+f^*L$ is an ample divisor on $X$. Then $$H=g^*(-K_X+f^*L)+\sum b_i E_i$$ is an ample divisor on $\tilde{X}$ for some rational numbers $b_i$ ( Such $L$ exists since $X \to Y$ is a Mori fiber space). Note that $$H\vert_{\tilde{Z}}=-K_Z +(\sum b_i E_i)\vert_Z+\text{fiber classes}.$$

We now specify the constraints. For each point class, we choose a point in $X$. For each $A_i$, let $B_i$ be the inverse images of $A_i$ in $\tilde{X}$. Then $[B_i]=(f\circ g)^*[A_i]$. Choose a very ample divisor $M$ on $\tilde{X}$. We will consider the Gromov-Witten invariant 
\[
\langle \underbrace{[pt], \ldots, [pt]}_{k}, [B_1], \dots, [B_n], \underbrace{[M^{d}], \ldots, [M^d]}_{k'}, \underbrace{[L]\cdot [M], \ldots, [L] \cdot [M]}_{k''} \rangle^X_{0, [\tilde{C}]},
\]
 where $d$ is the dimension of a general fiber and $\tilde{C}$ is a curve in $\tilde{X}$ such that $\pi_*([\tilde{C}])=[C]$.

Any stable map to $\tilde{X}$ of class $[\tilde{C}]$ meeting $k$ general points and $B_i$ gives a stable map to $Y$ of class $[C]$ meeting $k$ general points and $A_i$, thus has to lie in $\tilde{Z}$. We need to show curves in $\tilde{Z}$ meeting $k+k'$ general points in $\tilde{Z}$ and $k''$ codimension $2$ subvarieites in $\tilde{Z}$ are irreducible when $k'$ is large. This essentially follows from the proof of Theorem \ref{thm:fano}. 

The issue here is that the induced map on homology $i_*: H_2(\tilde{Z}, \ZZ) \to H_2(\tilde{X}, \ZZ)$ is not injective. So we have to show that all the curves in $\tilde{Z}$, whose curve class under $i_*$ is $[\tilde{C}]$, and meeting all the constraints are irreducible. As noted in Remark \ref{rem}, the proof of Theorem \ref{thm:fano} only uses the fact that the curve has a fixed intersection number with $\sum a_i E_i\vert_{\tilde{Z}}$, $K_{\tilde{Z}}$, and $H$. Since the curve class under $i_*$ is fixed, the intersection number only depends on the curve class $[\tilde{C}]$. Thus all such curves are irreducible.
\end{proof}

\begin{proof}[Proof of Theorem \ref{thm:section}]
By Theorem \ref{thm:generic}, we only need to find one smooth projective model of the generic fiber with the desired property. First note that we can refine the sequence
\[
\mcX_\eta=\mcX_n \to \mcX_{n-1} \to \ldots \to \mcX_1 \to \SP K(C)
\]
such that each morphism is a contraction of a extremal ray in the sense of MMP. We proceed by induction to construct a sequence of varieties 
\[
\xymatrix{
X_n \ar[r] \ar[d] &X_{n-1} \ar[r] \ar[d] &\ldots \ar[r] \ar[d] &X_1 \ar[r] \ar[d] &C \ar[d]\\
Y_n \ar[r] \ar[ru] &Y_{n-1} \ar[r] \ar[ru]  &\ldots \ar[r] \ar[ru]  &Y_1 \ar[r] & C}
\]
such that
\begin{enumerate}
\item The generic fibers of $X_i \to C$ and $Y_i \to C$ are isomorphic to $\mcX_i$.
\item Each variety $X_i$ is smooth and $Y_i$ has at worst terminal singularities.
\item Either $Y_i \to X_{i-1}$ is a Mori fiber space or birational, depending on whether $\mcX_i \to \mcX_{i-1}$ is a fiber type contraction or a birational morphism.
\item There exist birational morphisms $X_i \to Y_i$, which are isomorphisms over the smooth locus of $Y_i$.
\end{enumerate}

To start, choose a smooth projective model $Z_1$ of $\mcX_1$ over $C$. Then run the relative MMP (with scaling) of $Z_1$ over $C$. Denote the end product by $Y_1$. Take $X_1$ to be a resolution of singularities of $Y_1$ which is an isomorphism over the smooth locus of $Y_1$. Assume $X_i$ and $Y_i$ have been constructed. If $\mcX_{i+1} \to \mcX_i$ is birational, then take $X _{i+1} \cong Y_{i+1}$ to be any smooth projective model of $\mcX_{i+1}$ with a morphism to $X_i$. If $\mcX_{i+1} \to \mcX_i$ is a fiber type contraction, choose any smooth projective model $Z_{i+1}\to X_{i}$ and run the relative MMP (with scaling) over $X_i$. Take $Y_{i+1}$ to be the end product and $X_{i+1}$ to be a resolution of singularities of $Y_{i+1}$ which is isomorphic over the smooth locus of $Y_{i+1}$. Note that $Y_{i+1}$ has at worst terminal singularities.  

Now the theorem follows from Theorem \ref{thm:fano}, Lemma \ref{lem:birational} and \ref{lem:mfs}.
\end{proof}

\begin{proof}[Proof of Theorem \ref{thm:totalspace}]
Similar to the proof of Theorem \ref{thm:section}, we can construct a birational model $X' \to Y$ with a non-zero enumerative Gromov-Witten invariant of the form $\langle [pt], [pt], \pi^*[A_1], \ldots, \pi^*[A_n], \ldots \rangle^{X}$ by Lemma \ref{lem:birational} and \ref{lem:mfs}. Furthermore, we may assume the birational map $X \dashrightarrow X'$ is well-defined along a general fiber. We can successively blow-up smooth subvarieties to get a new variety 
\[
X_n \to X_{n-1} \to \ldots \to X_1 \to X_0=X
\]
 with a birational morphism $X_n \to X'$. Thus by Lemma \ref{lem:birational} and \ref{lem:mfs}, 
\begin{enumerate}
\item There is a non-zero Gromov-Witten invariant on $X_n$ of the form
\[
\langle [pt], [pt], \pi^*[A_1], \ldots, \pi^*[A_n], \ldots\rangle^{X_n}
\]
\item The push-forward of the corresponding curve class under the morphism $\pi_n: X_n \to X \to Y$ is the curve class $\beta$ giving the non-zero Gromov-Witten invariant on $Y$. 
\end{enumerate}
Thus it suffices to show these two conditions are preserved under the blow-downs $X_{i+1} \to X_{i}$.

Let $S_{i}$ be the blow-up center on $X_i$ and $E_{i+1}$ be the exceptional divisor on $X_{i+1}$. Since $X \dashrightarrow X'$ is well-defined along a general fiber, we may choose $S_{i}$ in such a way that the image of $S_i$ under $\pi_i: X_i \to Y$ is not dominant. We first apply the degeneration formula to the deformation to the normal cone of $E_{i+1}$. We specialize the constraints $[pt], [pt]$ to the $(X_{i+1}, E_{i+1})$ side. Thus we have
\begin{align*}
&\langle [pt], [pt], \pi^*[A_1], \ldots, \pi^*[A_n], \ldots \rangle^{X_{i+1}}\\
=&\sum \langle [pt], [pt], \ldots \vert \mcT \rangle^{(X_i, E_i)}\langle \ldots \vert \check{\mcT}\rangle^{(\PP_{E_i}(\OO_{E_i}(E_i)\oplus\OO_{E_i}), E_i)}
\end{align*}
Any non-zero term in the relative invariants side give rise to a (possibly reducible) curve $C_i$ in $X_i$ such that 
\begin{enumerate}
\item ${\pi_i}_*[C_i]=\beta$, where $\pi_i: X_i \to Y$ is the projection to $Y$,
\item The image of $C_i$ under $\pi_i$ in $Y$ meet two general points and all the representatives of $A_k$.
\end{enumerate}
 We may choose the representatives of $[A_k]$ to be general so that any curve in $Y$ of class $\beta$ and meeting $2$ general points and $A_k$'s are irreducible embedded and do not pass through the intersection of $\pi_i(S_i)$ with any of the $A_k$'s by the enumerative assumption. Thus the curve $C_i$ in $X_i$ contains one irreducible component whose image in $Y$ gives the irreducible curve in $Y$ meeting the constraints and the rest of the irreducible components of $C_i$ are all mapped to points in $Y$. And all the non-zero relative Gromov-Witten invariant on the $(X_i, E_i)$ side are connected invariants of the form
\[
\langle [pt], [pt], \pi^*[A_1], \ldots, \pi^*[A_n], \ldots \vert \mcT \rangle^{(X_i, E_i)}_{0, [C]},
\]
and $\pi_*([C])=\beta$.

Now apply Theorem \ref{thm:Correspondence} with the lowest order term on the relative invariant side. And we conclude that the two conditions are preserved under blow-downs.
\end{proof}

In order to prove Theorem \ref{thm:uniruled}, we need the following result.
\begin{thm}[\cite{KollarUni}, Theorem 4.2.10, \cite{RuanUni}, Proposition 4.9]\label{thm:kr}
Let $\bar{\Sigma}$ be a smooth projective uniruled variety. Then there is a non-zero Gromov-Witten invariant of the form $\langle [pt], \ldots \rangle^{\bar{\Sigma}}_{0, \beta}$.
\end{thm}

\begin{proof}
We first choose a polarization of $\bar{\Sigma}$. Then there exists a free curve $C$ of minimal degree with respect to the polarization. Note that every rational curve through a very general point $p$ in $\bar{\Sigma}$ is free. So if we choose such a point and consider all the curves mapping to $\bar{\Sigma}$ of class $[C]$ and passing through $p$, then we get a proper family ( by minimality) of expected dimension ( the deformation is unobstructed). Therefore the Gromov-Witten invariant $\langle [pt], [A]^2, \ldots [A]^2\rangle_{0, [C]}^{\bar{\Sigma}}$ is non-zero, where $[A]$ is the class of a very ample divisor. Clearly this is the number of curves meeting all the constraints.
\end{proof}

\begin{rem}
If the minimal free curve in $\bar{\Sigma}$ is an embedded curve, then this Gromov-Witten invariant is enumerative in the sense of Definition \ref{def:enumerative}.
\end{rem}

\begin{proof}[Proof of Theorem \ref{thm:uniruled}]
The assumptions implies that there is a compactification $\bar{\Sigma}$ of $\Sigma$ which has an enumerative Gromov-Witten invariant of the form $\langle [pt], \ldots \rangle^{\bar{\Sigma}}_{0, \beta}$. By Lemma \ref{lem:mfs} and \ref{lem:birational}, there is a birational model $X'$ of $X$, which has a morphism to $\bar{\Sigma}$, isomorphic to $X$ near the generic fiber of $X \dashrightarrow \Sigma$, and has a non-zero Gromov-Witten invariant of the form $\langle [pt], \ldots \rangle^{X'}_{0, \beta'}$. Furthermore, $\beta'$ is not a class supported in a fiber.

By the weak factorization theorem (\cite{AKMW}, Theorem 0.1.1), we may factorize the birational map between $X$ and $X'$ by a number of blow-up/blow-down along smooth centers away from a general fiber. So it suffices to show that a non-zero invariant of this form is preserved under blow-ups and blow-downs and the image of the curve in $\Sigma$ is not a point (or equivalently, the curve is not contained in a fiber). The invariance of such a non-zero Gromov-Witten invariant is proved by \cite{HLR} via the degeneration formula and Theorem \ref{thm:Correspondence}(Theorem 5.15 in \cite{HLR}). Notice that in the course of applying the degeneration formula, we start with a curve whose image is not a point in $\Sigma$. Thus the curve class of the invariant in the blow-up/blow-down is mapped to a point only if it comes from some disconnected invariant, which is simply the product of connected Gromov-Witten invariants of connected components. But in that case there are cohomology classes in the insertions of the disconnected Gromov-Witten invariant coming from cohomologies of the blow-up center. The curve in this class has to meet a general point and representatives of these classes, which can be chosen to be away from a general fiber. So the image of the curve in $\Sigma$ cannot be a point.
\end{proof}

\begin{rem}\label{rem:immersedcase}
The condition on $\Sigma$ of having a birational model with an embedded minimal free rational curve can be removed when the general fiber is $\PP^1$. In this case, a general minimal free curve is immersed. And the total space of the fibration over a general such curve is a non-normal surface which is singular only along smooth fibers, and whose normaliztion is smooth. It is easy to modify the proof to adapt to this case.
\end{rem}
%%%%%%%%%%%%%%%%%%%%%%%%%%%%%%%%%%%%%%%%%%%%%%%%%%%%%%%%%%%%%%%%
\bibliographystyle{alpha}
\bibliography{MyBib}

\end{document}